\documentclass[12pt]{article}
\usepackage[a4paper,width=6.8in,height=8.8in]{geometry}
\usepackage{amsmath,amssymb}

\usepackage{amsmath,amssymb,mathtools,amsthm,
	textcomp,mathscinet}
\usepackage{amsfonts,graphicx,enumerate,subcaption,placeins}
\usepackage[hidelinks]{hyperref}
\usepackage{setspace}
\usepackage[utf8]{inputenc}
\usepackage[english]{babel}
\usepackage{booktabs,multirow}
\usepackage{blkarray}
\usepackage{lscape}
\usepackage[T1]{fontenc}
\usepackage{authblk,color}
\usepackage[english]{babel}
\usepackage{amsfonts, amstext, amsthm, booktabs, dcolumn}
\usepackage{graphicx}
\usepackage{float}
\usepackage{caption}
\usepackage{centernot}
\usepackage{hyperref}
\usepackage{times}
\usepackage[nottoc]{tocbibind}
\usepackage[mathscr]{eucal}

\definecolor{gr}{rgb}{0.7, 0.0, 0.15}

\newtheorem{theorem}{\bf Theorem}[section]

\newtheorem{corollary}{\bf Corollary}[section]
\newtheorem{remark}{\bf Remark}[section]

\newtheorem{lemma}{\bf Lemma}[section]

\numberwithin{equation}{section}
\begin{document}
\title{Distribution of a Unified $(k_1,k_2,\ldots,k_m)$-run}
\author[1]{{\Large Savita Chaturvedi}}
\author[2]{{\Large Amit N. Kumar}}
%\author[3]{{\Large Serkan Eryilmaz}}
%\author[4]{{\Large Stathis Chadjiconstantinidis}}
\affil[1,2]{Department of Mathematical Sciences, Indian Institute of Technology (BHU), Varanasi (Uttar Pradesh) - 221 005, India.}
%\affil[3]{Department of Industrial Engineering, Atilim University, Ankara, Turkey.}
%\affil[4]{Department of Statistics and Insurance Science, University of Piraeus, Piraeus, Greece.}
\affil[1]{Email: savitachaturvedi.rs.mat23@itbhu.ac.in}
\affil[2]{Email: amit.mat@itbhu.ac.in}
%\affil[3]{Email: serkan.eryilmaz@atilim.edu.tr}
%\affil[4]{Email: stch@unipi.gr}
\date{}
\maketitle

\begin{abstract}
\noindent
We explore a unified $(k_1,k_2,\ldots,k_m)$-run in multi-state trials, examining its distributional properties and waiting time distribution. Our study reveals that this particular run serves as a generalization encompassing various patterns. Additionally, we discuss various results pertaining to existing patterns as special cases. To illustrate our findings, we provide an application related to DNA frequent patterns.
\end{abstract}

\noindent

\section{Introduction and Preliminaries}\label{11:sec1}
The concept of runs has extensive applications across various fields such as climatology, quality control, biology, computer science, and DNA sequence analysis, among many others. In a sequence of Bernoulli trials with possible outcomes of failure and success, the distribution of the number of occurrences of $k$-runs (defined as at least $k$ consecutive successes) holds significant importance and has been extensively discussed in the literature. Several authors have studied the distributional properties of $k$-runs and its waiting time distribution along with various applications, see for more details, Aki {\em et al.} \cite{AKH}, Antzoulakos {\em et al.} \cite{ABK}, Antzoulakos and Chadjiconstantinidis \cite{AC}, Balakrishnan and Koutras \cite{BK}, Fu and Koutras \cite{FK}, Makri {\em et al.} \cite{MPP}, Philippou {\em et al.} \cite{PGP}, and references therein. Additionally, the research on approximations related to runs has also been extensively explored in the literature. For more details, see \v{C}ekanavi\v{c}ius and Vellaisamy \cite{CV2015,CV2019,CV2020}, Fu and Johnson \cite{FJ}, Kumar and Kumar \cite{KK2024}, Kumar and Upadhye \cite{KU, KU2022}, Kumar {\em et al.} \cite{KUV}, Liaudanskait\.{e} and \v{C}ekanavi\v{c}ius \cite{LC2022}, Upadhye and Kumar \cite{UK2018}, Upadhye {\em et al.} \cite{UCV}, Vellaisamy \cite{V}, Wang and Xia \cite{WX}, and references therein.

\noindent
Huang and Tsai \cite{HT} extended the concept of $k$-runs by introducing the notion of $(k_1, k_2)$-runs, which involve the pattern at least $k_1$ consecutive failures followed by at least $k_2$ consecutive successes. Subsequently, Dafnis {\em et al.} \cite{DAP} explored three variations of $(k_1,k_2)$-runs, considering scenarios involving at least, exactly, and at most occurrences of runs. Further, Kumar and Upadhye \cite{KU2019} generalized these patterns to encompass three general types of runs. While previous studies focused on two-state trials in studying $k$-runs and $(k_1,k_2)$-runs, the demand across various applications necessitates the exploration of $(k_1,k_2,\ldots,k_m)$-runs in multi-state trials. For instance, in the DNA sequence analysis, it is crucial to represent sequences using four letters: A, G, C, and T, which correspond to a sequence of four-state trials. Kong \cite{Kong2021} initiated this study by focusing the run at least $k_1$ consecutive 1's followed by $k_2$ consecutive 2's $\ldots$ followed by $k_m$ consecutive $m$'s in a sequence of multi-state trials with labels $\{1,2,\ldots,m\}$. Mathematically, let $\xi_1,\xi_2,\ldots,\xi_n$ be a sequence of independent and identically distributed (\textit{iid}) multinomial random variables, each taking values from the set $\{1,2,\ldots,m\}$, with probability $\mathbb{P}(\xi_i=j)=p_j$, for $j=1,2,\ldots,m$, and $\sum_{j=1}^{m}p_j=1$. Define
\begin{align*}
    \mathbb{I}_i:=\left(\prod_{j_1=i}^{i+k_1-1}(2-\xi_{j_1})\right)\left(\prod_{j_2=i+k_1}^{i+k_1+k_2-1}(3-\xi_{j_2})\right)\ldots \left(\prod_{j_m=i+k_1+\cdots+k_{m-1}}^{i+k_1+\cdots+k_m-1}(m+1-\xi_{j_m})\right).
\end{align*}
Then, 
\begin{align*}
    X_{m;k_1,k_2,\ldots,k_m}^{(n)}:=\sum_{i=1}^{n-k_1-\cdots-k_m+1}\mathbb{I}_i
\end{align*}
represents the number of occurrences of the considered $(k_1,k_2,\ldots,k_m)$-run. Kong \cite{Kong2021} developed a generating function approach to derive the exact distribution of $X_{m;k_1,k_2,\ldots,k_m}^{(n)}$. Additionally, Chadjiconstantinidis and Eryilmaz \cite{CE2023} have studied the waiting time distribution of $X_{m;k_1,k_2,\ldots,k_m}^{(n)}$.

\noindent
In this paper, we explore a unified $(k_1,k_2,\ldots,k_m)$-run in a sequence of multi-state trials with labels $\{1,2,\ldots,m\}$ defined as follows:
\begin{enumerate}
    \item[\textbf{(R)}] at least $\ell_1$ and at most $k_1$ consecutive $1$'s followed by at least $\ell_2$ and at most $k_2$ consecutive $2$'s $\ldots$ followed at least $\ell_m$ and at most $k_m$ consecutive $m$'s.
\end{enumerate}

\noindent
Here, we assume $k_i\ge\ell_i\ge 1$.\\
Mathematically, let
\begin{align*}
    \mathbb{I}_{i;r_1,r_2,\ldots,r_m}:=\left(\prod_{j_1=i}^{i+r_1-1}(2-\xi_{j_1})\right)\left(\prod_{j_2=i+r_1}^{i+r_1+r_2-1}(3-\xi_{j_2})\right)\ldots \left(\prod_{j_m=i+r_1+\cdots+r_{m-1}}^{i+r_1+\cdots+r_m-1}(m+1-\xi_{j_m})\right),
\end{align*}
where $r_i=s_i+\ell_i$, for $i=1,2,\ldots,m$, and $\displaystyle{\mathbb{I}_i^*:=\max_{\substack{ r_i\le \ell_i+k_i\\i=1,2,\ldots,m}}\mathbb{I}_{i;r_1,r_2,\ldots,r_m}}$. Then
\begin{align*}
    X_{m;\ell_1:k_1,\ell_2:k_2,\ldots,\ell_m:k_m}^{(n)}:=\sum_{i=1}^{n-\ell_1-\cdots-\ell_m+1}\mathbb{I}_i^*
\end{align*}
represents the number of occurrences of the run \textbf{(R)}. To simplify the presentation of paper,  we introduce the following notations:
\begin{itemize}
\item $E_{k_i}$ denotes a run of at least $k_i$ consecutive $i$'s.
\item $F_{k_i}$ denotes a run of exactly $k_i$ consecutive $i$'s.
\item $G_{k_i}$ denotes a run of at most $k_i$ consecutive $i$'s.
\item $E_{\ell_i,k_i}$ denotes a run of at least $\ell_i$ and at most $k_i$ consecutive $i$'s.
\end{itemize}
Moreover, we define the notation $E_{i} \prec E_{j}$ to represent the event $E_{i}$ appears immediately after the event $E_{j}$. Therefore, the run \textbf{(R)} is equivalent to $E_{\ell_1,k_1} \prec E_{\ell_2,k_2} \prec \cdots \prec E_{\ell_m,k_m}$. Note that if $\ell_i=k_i$ and $\ell_i=1$, for $i=1,2,\ldots,m$, then $E_{\ell_i,k_i}=F_{k_i}$ and $E_{\ell_i,k_i}=G_{k_i}$, respectively. In such cases, the run \textbf{(R)} simplifies to $F_{k_1} \prec F_{k_2} \prec \cdots \prec F_{k_m}$ and $G_{k_1} \prec G_{k_2} \prec \cdots \prec G_{k_m}$, respectively. It is worth noting that $E_{\ell_i,k_i}\to E_{\ell_i}$ when $k_i\to \infty$. Therefore, if $k_i\to \infty$, for $i=1,2,\ldots,m$, then the run \textbf{(R)} reduces to $E_{\ell_1} \prec E_{\ell_2} \prec \cdots \prec E_{\ell_m}$, which is studied by by Kong \cite{Kong2021}. Thus, the run \textbf{(R)} can be transformed into a $(k_1,k_2,\ldots,k_m)$-run with any $k_i$, for $i=1,2,\ldots,m$, defined as at least, exactly, and at most $k_i$ consecutive $i$'s. Therefore, the study of the run \textbf{(R)} suffices to study any type of $(k_1,k_2,\ldots,k_m)$-run with all at least, exactly, and at most cases.

\noindent
We employ a modified methodology introduced by Kong \cite{Kong2001,Kong2021} to derive our results. The advantage of this method lies in its utilization of the generating functions of individual patterns to derive the generating function for complicated patterns of interest. The technique can be formulated as follows: Suppose a linear sequence comprising multiple objects with values drawn from the set $\{1,2,\ldots,m\}$, where objects may occur repeatedly. Subsequently, we partition the sequence into distinct patterns of our interest, each forming specific blocks. These blocks are organized based on the pattern we wish to explore, ensuring relevance and coherence throughout the analysis. The relationship between blocks is characterized by the interaction $w_{ij}$, representing the influence of the $i$th type of block on the left and the $j$th type of block on the right. Furthermore, we denote the generating function as $g_i$ associated with block $i$. The generating function for the entire system can be expressed as 
\begin{align}
    G=1+\boldsymbol{eM}^{-1}\boldsymbol{g}.\label{11:eq1}
\end{align}
Here, $\boldsymbol{e}=(1~1~\cdots~1)$ represents a row vector with all entries equal to $1$, $\boldsymbol{g}=(g_1 ~ g_2 ~ \cdots ~ g_r)^{T}$ represents a column vector, where each entry corresponds to a generating function $g_i$, $i=1,2,\ldots,r$, and 
\begin{align*}
    \boldsymbol{M}:=\begin{pmatrix}
        1 & -g_1 w_{12} & -g_1 w_{13} & \cdots & -g_1 w_{1r}\\
        -g_2 w_{21} & 1 & -g_2 w_{23} & \cdots & -g_2 w_{2r}\\
        \vdots & \vdots & \vdots & \ddots & \vdots\\
        -g_r w_{r1} & -g_r w_{r2} & -g_r w_{r3} & \cdots & 1
    \end{pmatrix}.
\end{align*}
When blocks are forbidden, the interaction is set to zero. To compute the distribution of $(k_1, k_2, ..., k_m)$-runs, one approach is to initially consider two consecutive blocks together. Subsequently, by iteratively adding more blocks, we can obtain the generating function for $(k_1, k_2, ..., k_m)$-runs. For further details,  we refer the reader to Kong \cite{Kong2001,Kong2006,Kong2021}.

\noindent
This paper is organized as follows: In Section \ref{11:sec2}, we derive the distributional properties of the random variable $X_{m;\ell_1:k_1,\ell_2:k_2,\ldots,\ell_m:k_m}^{(n)}$. Our study reveals that this random variable serves as a generalization encompassing various patterns. Additionally, we discuss various results pertaining to existing patterns as special cases. In Section \ref{11:sec3}, we further explore the distributional properties of the waiting time distribution of $X_{m;\ell_1:k_1,\ell_2:k_2,\ldots,\ell_m:k_m}^{(n)}$. Finally, in Section \ref{11:sec4}, we present an application to DNA frequent pattern, illustrating the practical implications of our results.

\section{Distribution of $\boldsymbol{X_{m;\ell_1:k_1,\ell_2:k_2,\ldots,\ell_m:k_m}^{(n)}}$}\label{11:sec2}
In this section, we explore the run \textbf{(R)} introduced in Section \ref{11:sec1} and derive its distributional properties using the generating function approach proposed by Kong \cite{Kong2021}. Our proofs are established through induction. Initially, we focus on the case of two-state trials, presenting several findings related to previous studies. Subsequently, we extend our analysis to multi-state trials, where we demonstrate the applicability of our results and how they encompass existing findings as special cases. 

\noindent
We begin to introduce some notations that will be useful for further finding. Throughout this paper, let $m\ge 2$, $\prod_{i}^{j}=1$ and $\sum_{i}^{j}=0$, for $j<i$. Recall that $\xi_1,\xi_2,\ldots,\xi_n$ is a sequence of \textit{iid} multinomial random variables, each taking values from the set $\{1,2,\ldots,m\}$, with probability $\mathbb{P}(\xi_i=j)=p_j$, for $i=1,2,\ldots,n$, $j=1,2,\ldots,m$, and $\sum_{j=1}^{m}p_j=1$. The random variable $X_{m;\ell_1:k_1,\ell_2:k_2,\ldots,\ell_m:k_m}^{(n)}$ denotes the number of occurrences of the run $E_{\ell_1,k_1} \prec E_{\ell_2,k_2} \prec \cdots \prec E_{\ell_m,k_m}$. Further, let us denote $p_{m,n}(r)=\mathbb{P}(X_{m;\ell_1:k_1,\ell_2:k_2,\ldots,\ell_m:k_m}^{(n)}=r)$. The generating function for the sequence of length $n$ is represented as
\begin{align*}
    \phi_{m,n}(w)=\sum_{j=0}p_{m,n}(j)w^j
\end{align*}
and the double generating function is expressed as
\begin{align}
    \Phi_m(w,z)=\sum_{n=0}^{\infty}\phi_{m,n}(w)z^n=\sum_{n=0}^{\infty}\sum_{j=0}p_{m,n}(j)w^jz^n.\label{11:dgf}
\end{align}
Next, we introduce some generating functions as follows:
\begin{align*}
    g_i^{(1)}&=\sum_{j=1}^{\ell_i-1}(p_i z)^j=\frac{p_i z-(p_i z)^{\ell_i}}{1-p_i z}, \quad i=1,2,\ldots,m,\\
      g_i^{(2)}&=\sum_{j=k_1+1}^{\infty}(p_i z)^j=\frac{(p_i z)^{k_i+1}}{1-p_i z}, \quad i=1,2,\ldots,m,\\
      g_i^{(3)}&=\sum_{j=\ell_i}^{k_i}(p_i z)^j=\frac{(p_i z)^{\ell_i}-(p_i z)^{k_i+1}}{1-p_i z}, \quad i=1,2,\ldots,m.
\end{align*}
Here, we assume $p_iz<1$ for all $i=1,2,\ldots,m$. The following result provides the double generating function for the probabilities of the  run $E_{\ell_1,k_1} \prec E_{\ell_2,k_2}$.

\begin{theorem}\label{11:th1}
    The double generating function for the probabilities of the run \textbf{(R)} with $m=2$ is given by
    \begin{align}
        \Phi_2(w,z)=\frac{1}{1-z-(w-1)((p_1z)^{\ell_1}-(p_1 z)^{k_1+1})((p_2z)^{\ell_2}-(p_2 z)^{k_2+1})}.\label{11:dgf1}
    \end{align}
\end{theorem}
\begin{proof}
Let us divide the sequence into three distinct blocks as follows: the first block consists of at most $\ell_i-1$ consecutive $i$'s, the second block consists of at least $k_i+1$ consecutive $i$'s, and the third block consists of at least $\ell_i$ to at most $k_i$ consecutive $i$'s, for $i=1,2$. Then the generating functions for these blocks are $g_i^{(1)}$, $g_i^{(2)}$, and $g_i^{(3)}$, respectively. Thus, $\boldsymbol{e}=[1~1~1~1~1~1]$,
\begin{align*}
    \boldsymbol{M}=\begin{pmatrix}
        1 & 0 & 0 & -g_1^{(1)} & -g_1^{(1)} &-g_1^{(1)}\\
        0 & 1 & 0 & -g_1^{(2)} & -g_1^{(2)} &-g_1^{(2)}\\
        0 & 0 & 1 & -g_1^{(3)} & -g_1^{(3)} &-g_1^{(3)}w\\
        -g_2^{(1)} & -g_2^{(1)} & -g_2^{(1)} & 1 & 0 & 0\\
        -g_2^{(2)} & -g_2^{(2)} & -g_2^{(2)} & 0 & 1 & 0\\
        -g_2^{(3)} & -g_2^{(3)} & -g_2^{(3)} & 0 & 0 & 1\\
    \end{pmatrix},
\end{align*}
 and $\boldsymbol{g}=[g_1^{(1)}~g_1^{(2)}~g_1^{(3)}~g_2^{(1)}~g_2^{(2)}~g_2^{(3)}]^T$. Observe that the interaction is defined based on the adjacency of the blocks, where $w_{12}=0$ indicates that these two blocks are forbidden to be adjacent to each other. The same  argument applies to cases where the interaction is zero. Additionally, $w_{36}=w$ is used to track the number patterns of the run \textbf{(R)} with $m=2$. Hence, substituting $\boldsymbol{e}$, $\boldsymbol{M}$ and $\boldsymbol{g}$ in \eqref{11:eq1}, the result follows.
\end{proof}

\begin{remark}
    Note that the double generating function presented in \eqref{11:dgf1} differs from the one derived in (7) of Kumar and Upadhye \cite{KU2019}.. The disparity arises from their consideration of a failure following the pattern $E_{\ell_1,k_1}\prec E_{\ell_2,k_2}$. A similar argument holds to the pattern $E_{\ell_1}\prec E_{\ell_2,k_2}$.
\end{remark}

\noindent
Next, the following corollary can be straightforwardly derived by letting $k_i \to \infty$, $\ell_i=k_i$, and $\ell_i=1$ for $i = 1, 2$, in \eqref{11:dgf1}.

\begin{corollary}
\begin{itemize}
    \item[(i)] The double generating function for the probabilities of the run $E_{\ell_1}\prec E_{\ell_2}$ is given by
       \begin{align}
        \Phi_2(w,z)=\frac{1}{1-z-(w-1)(p_1z)^{\ell_1}(p_2z)^{\ell_2}}.\label{11:eq2}
    \end{align}
    \item[(ii)] The double generating function for the probabilities of the run $F_{k_1}\prec F_{k_2}$ is given by
       \begin{align}
        \Phi_2(w,z)=\frac{1}{1-z-(w-1)(1-p_1z)(1-p_2z)(p_1z)^{k_1}(p_2z)^{k_2}}.\label{11:eq3}
    \end{align}
    \item[(iii)] The double generating function for the probabilities of the run $G_{k_1}\prec G_{k_2}$ is given by
       \begin{align}
        \Phi_2(w,z)=\frac{1}{1-z-(w-1)(p_1z)(p_2z)(1-(p_1z)^{k_1})(1-(p_2z)^{k_2})}.\label{11:eq4}
    \end{align}
    \end{itemize}
\end{corollary}

\begin{remark}
   It is worth noting that the double generating functions provided in \eqref{11:eq2}, \eqref{11:eq3}, and \eqref{11:eq4} coincide with those derived by Dafnis {et al.} \cite{DAP} in (3.1), (3.2), and (3.3), respectively. 
\end{remark}

\noindent
Further, the following corollary can be obtained by taking $k_2 \to \infty$ in \eqref{11:dgf1}.

\begin{corollary}
The double generating function for the probabilities of the run $E_{\ell_1,k_1}\prec E_{\ell_2}$ is given by
       \begin{align}
        \Phi_2(w,z)=\frac{1}{1-z-(w-1)(p_1z)^{\ell_1}(p_2z)^{\ell_2}(1-(p_1 z)^{k_1-\ell_1+1})}.\label{11:eq5}
    \end{align}
\end{corollary}

\begin{remark}
   Note that the double generating functions presented in \eqref{11:eq5} coincide with one derived by Kumar and Upadhye \cite{KU2019} in (3). Moreover, the obtained result in Theorem \ref{11:th1} extends to yield the double generating function for any type of $(k_1,k_2)$-runs with all at least, exactly, and at most cases.
\end{remark}

\noindent
It is important to highlight the advantage of this technique in connecting one pattern with another. If we aim to add one more pattern to previously obtained patterns, then information about the generating function of the pattern appearing at the right end of the sequence becomes crucial. This allows us to establish connections between adjacent blocks, facilitating the derivation of generating functions for more generalized patterns. For instance, in determining the generating function of the pattern $E_{\ell_1,k_1}\prec E_{\ell_2,k_2}\prec E_{\ell_3,k_3}$, we require the generating functions of the pattern $E_{\ell_1,k_1}\prec E_{\ell_2,k_2}$ appearing at the end of the sequence. Furthermore, we can derive the generating function of the run $E_{\ell_1,k_1}\prec E_{\ell_2,k_2}\prec E_{\ell_3,k_3}$ by adding $E_{\ell_3,k_3}$ to the right end sequence of $E_{\ell_1,k_1}\prec E_{\ell_2,k_2}$. This iterative process continues until we obtain the generating function of the pattern $E_{\ell_1,k_1}\prec E_{\ell_2,k_2}\prec \ldots \prec E_{\ell_m,k_m}$. Hence, our initial focus is on computing the generating functions of the pattern \textbf{(R)} at the right end of the sequence. The following theorem will provide the generating function of the pattern \textbf{(R)} that appears to the right end of the sequence.

\begin{theorem}\label{11:th2}
   The generating function for the probabilities of the pattern $E_{\ell_1,k_1}\prec E_{\ell_2,k_2}\prec \ldots \prec E_{\ell_m,k_m}$ appearing at the right end of the sequence is given by
   \begin{align}
       Y_m(z)=\frac{\prod_{i=1}^{m}((p_i z)^{\ell_i}-(p_i z)^{k_i+1})}{\left(1-z\sum_{i=1}^{m}p_i \right)\prod_{i=2}^{m}(1-p_i z)}.\label{11:eq12}
   \end{align}
  The generating function for the complement of patterns that do not contain such a sequence is given by
   \begin{align*}
       N_m(z)=\frac{z\left(\sum_{i=1}^{m}p_i\right)\prod_{i=2}^{m}(1-p_i z)-\prod_{i=1}^{m}((p_i z)^{\ell_i}-(p_i z)^{k_i+1})}{\left(1-z \sum_{i=1}^{m}p_i\right)\prod_{i=2}^{m}(1-p_i z)}.
   \end{align*}
\end{theorem}

\begin{proof}
    We prove the result using induction of $m$. For $m=2$, we have $\boldsymbol{e}=[1~1~1~1~1~1~1]$,
    \begin{align*}
        \boldsymbol{M}:=\begin{pmatrix}
            1 & 0 & 0 & -g_1^{(1)} & -g_1^{(1)} & -g_1^{(1)} & -g_1^{(1)}\\
            0 & 1 & 0 & -g_1^{(2)} & -g_1^{(2)} & -g_1^{(2)} & -g_1^{(2)}\\
            0 & 0 & 1 & -g_1^{(3)} & -g_1^{(3)} & -g_1^{(3)} & -g_1^{(3)}u\\
            -g_2^{(1)} & -g_2^{(1)} & -g_2^{(1)} & 1 & 0 & 0 & 0\\
            -g_2^{(2)} & -g_2^{(2)} & -g_2^{(2)} & 0 & 1 & 0 & 0 \\
            -g_2^{(3)} & -g_2^{(3)} & -g_2^{(3)} & 0 & 0 & 0 & 1\\
            0 & 0 & 0 & 0 & 0 & 0 & 1\\
        \end{pmatrix},
    \end{align*}
    and $\boldsymbol{g}=[g_1^{(1)}~g_1^{(2)}~g_1^{(3)}~g_2^{(1)}~g_2^{(2)}~g_2^{(3)}~g_2^{(3)}]^T$. The matrix $\boldsymbol{M}$ is generated using a similar concept applied in Theorem \ref{11:th1}. Here, we introduce a special block $E_{\ell_3,k_3}$ that appears to the right of the sequence. It is important to note that this block has zero interaction with any other blocks. The variable $u$ is used to track the pattern $E_{\ell_1,k_1}$ from the left and $E_{\ell_2,k_2}$ from the right, with the additional requirement of the block at the right of the sequence. Computing $\boldsymbol{e}\boldsymbol{M}^{-1}\boldsymbol{g}$ will provide the generating function in terms of a polynomial in $u$, where the coefficient of $u$ yields to $Y_2(m)$. For more details, see Lemma 1 of Kong \cite{Kong2001}. With some routine calculations, the coefficient of $u$ in $\boldsymbol{e}\boldsymbol{M}^{-1}\boldsymbol{g}$ can be expressed as
     \begin{align*}
       Y_2(z)=\frac{\prod_{i=1}^{2}((p_i z)^{\ell_i}-(p_i z)^{k_i+1})}{\left(1-z \right)(1-p_2 z)}.
   \end{align*}
   The complement generating function $N_2(z)$ can be computed by subtracting $Y_2(z)$ from the full system. It is expressed as
   \begin{align*}
       N_2(z)=\frac{(p_1+p_2)z}{1-(p_1+p_2)z}-Y_2(z)=\frac{z(p_1+p_2)(1-p_2 z)-\prod_{i=1}^{m}((p_i z)^{\ell_i}-(p_i z)^{k_i+1})}{\left(1-z (p_1+p_2)\right)(1-p_2 z)}.
   \end{align*}
   Hence, the result is true for $m=2$. Assuming the results hold for $m=k-1$, let us consider $\boldsymbol{e}=[1~1~1~1~1~1]$,
   \begin{align*}
        \boldsymbol{M}:=\begin{pmatrix}
            1 & 0 & -N_{k-1}(z) & -N_{k-1}(z) & -N_{k-1}(z)  & -N_{k-1}(z)\\
            0 & 1 & -Y_{k-1}(z) & -Y_{k-1}(z) & -Y_{k-1}(z) & -Y_{k-1}(z)u\\
            -g_k^{(1)} & -g_k^{(1)} & 1 & 0 & 0 & 0\\
            -g_k^{(2)} & -g_k^{(2)} & 0 & 1 & 0 & 0 \\
            -g_k^{(3)} & -g_k^{(3)} & 0 & 0 & 1 & 0\\
            0 & 0 & 0 & 0 & 0 & 1\\
        \end{pmatrix},
    \end{align*}
    and $\boldsymbol{g}=\left[N_{k-1}(z)~~Y_{k-1}(z)~~g_k^{(1)}~~g_k^{(2)}~~g_k^{(3)}~~g_k^{(3)}\right]^T$. Following the steps similar to the standard case $m=2$, the coefficient of $u$ is given by
    \begin{align}
        Y_k(z)=\frac{g_k^{(3)} Y_{k-1}\left(1+g_{k}^{(1)}+g_{k}^{(2)}+g_{k}^{(3)}\right)}{1-\left(g_{k}^{(1)}+g_{k}^{(2)}+g_{k}^{(3)}\right)(Y_{k-1}(z)+N_{k-1}(z))}.\label{11:eq11}
    \end{align}
    The result can be verified for $m=k$ by substituting the following identities in \eqref{11:eq11}
    \begin{align*}
        g_{k}^{(1)}+g_{k}^{(2)}+g_{k}^{(3)}=\frac{p_k z}{1-p_k z}\quad \text{and}\quad Y_{k-1}(z)+N_{k-1}(z)=\frac{(p_1+p_2+\cdots+p_{k-1})z}{1-(p_1+p_2+\cdots+p_{k-1})z}.
    \end{align*}
    Hence, the results hold for any $m\ge 2$.
\end{proof}

\noindent
The following corollary can be easily verified by taking $k_i\to \infty$, for $i=1,2,\ldots,m$.

\begin{corollary}
     The generating function for the probabilities of the pattern $E_{\ell_1}\prec E_{\ell_2}\prec \ldots \prec E_{\ell_m}$ appearing at the right end of the sequence is given by
   \begin{align*}
       Y_m(z)=\frac{\prod_{i=1}^{m}(p_i z)^{\ell_i}}{\left(1-z\sum_{i=1}^{m}p_i \right)\prod_{i=2}^{m}(1-p_i z)}.
   \end{align*}
  The generating function for the complement of patterns that do not contain such a sequence is given by
   \begin{align*}
       N_m(z)=\frac{z\left(\sum_{i=1}^{m}p_i\right)\prod_{i=2}^{m}(1-p_i z)-\prod_{i=1}^{m}(p_i z)^{\ell_i}}{\left(1-z \sum_{i=1}^{m}p_i\right)\prod_{i=2}^{m}(1-p_i z)}.
   \end{align*}
\end{corollary}

\begin{remark}
    Note that the generating functions provided in \eqref{11:eq12} coincide with Lemma 1 of Kong \cite{Kong2021} for the pattern $E_{\ell_1}\prec E_{\ell_2}\prec \ldots \prec E_{\ell_m}$.
\end{remark}

\noindent
Finally, we are at a stage where we can derive our main result by using Theorems \ref{11:th1} and \ref{11:th2}. The following theorem provide the generating function for the pattern $E_{\ell_1,k_1}\prec E_{\ell_2,k_2}\prec \ldots \prec E_{\ell_m,k_m}$.

\begin{theorem}\label{11:th3}
    The double generating function for the probabilities of the run \textbf{(R)} is given by
    \begin{align*}
        \Phi_m(w,z)=\frac{\prod_{i=2}^{m-1}(1-p_i z)}{(1-z)\prod_{i=2}^{m-1}(1-p_iz)-(w-1)\prod_{i=1}^{m}((p_i z)^{\ell_i}-(p_i z)^{k_i+1})}.
    \end{align*}
\end{theorem}
\begin{proof}
    We prove the result using induction of $m$. We have already proved the result for $m=2$ in Theorem \ref{11:th1}. For $m=3$, let us consider $\boldsymbol{e}=[1~1~1~1~1]$, 
    \begin{align*}
        \boldsymbol{M}:=\begin{pmatrix}
            1 & 0 & -N_{2}(z) & -N_{2}(z) & -N_{2}(z)\\
            0 & 1 & -Y_{2}(z) & -Y_{2}(z) & -Y_{2}(z)w\\
            -g_3^{(1)} & -g_3^{(1)} & 1 & 0 & 0 \\
            -g_3^{(2)} & -g_3^{(2)} & 0 & 1 & 0  \\
            -g_3^{(3)} & -g_3^{(3)} & 0 & 0 & 1 \\
        \end{pmatrix},
    \end{align*}
    and $\boldsymbol{g}=\left[N_{2}(z)~~Y_{2}(z)~~g_3^{(1)}~~g_3^{(2)}~~g_3^{(3)}\right]^T$. Following the steps similar to Theorem \ref{11:th1}, we get 
    \begin{align*}
        \Phi_3(w,z)=\frac{(1-p_2 z)}{(1-z)(1-p_2z)-(w-1)\prod_{i=1}^{3}((p_i z)^{\ell_i}-(p_i z)^{k_i+1})}.
    \end{align*}
    Thus, the result is true for $m=3$. Assume the result holds for $m=k-1$. Now, consider $\boldsymbol{e}=[1~1~1~1~1]$, 
    \begin{align*}
        \boldsymbol{M}:=\begin{pmatrix}
            1 & 0 & -N_{k-1}(z) & -N_{k-1}(z) & -N_{k-1}(z)\\
            0 & 1 & -Y_{k-1}(z) & -Y_{k-1}(z) & -Y_{k-1}(z)w\\
            -g_k^{(1)} & -g_k^{(1)} & 1 & 0 & 0 \\
            -g_k^{(2)} & -g_k^{(2)} & 0 & 1 & 0  \\
            -g_k^{(3)} & -g_k^{(3)} & 0 & 0 & 1 \\
        \end{pmatrix},
    \end{align*}
    and $\boldsymbol{g}=\left[N_{k-1}(z)~~Y_{k-1}(z)~~g_k^{(1)}~~g_k^{(2)}~~g_k^{(3)}\right]^T$. Substituting $N_{k-1}(z)$ and $Y_{k-1}(z)$ from Theorem \ref{11:th2} and following the steps similar as above, the result follows for $m=k$. \\
    Hence, the result holds for any $m\ge 2$.
\end{proof}

\noindent
The following corollary can be easily verified by letting $k_i\to \infty$, for $i=1,2,\ldots,m$ in Theorem \ref{11:th3}.
\begin{corollary}\label{11:cor1}
    The double generating function for the probabilities of the run $E_{\ell_1}\prec E_{\ell_2}\prec \cdots \prec E_{\ell_m}$ is given by
    \begin{align}
        \Phi_m(w,z)=\frac{\prod_{i=2}^{m-1}(1-p_i z)}{(1-z)\prod_{i=2}^{m-1}(1-p_iz)-(w-1)\prod_{i=1}^{m}(p_i z)^{\ell_i}}.\label{11:eq21}
    \end{align}
\end{corollary}

\begin{remark}
    Observe that the generating functions provided in \eqref{11:eq21} coincide with Theorem 1 of Kong \cite{Kong2021} for the pattern $E_{\ell_1}\prec E_{\ell_2}\prec \ldots \prec E_{\ell_m}$.
\end{remark}

\noindent
As special cases, the following corollary can be proved by taking $\ell_i=k_i$ and $\ell_i=1$, for $i=1,2,\ldots,m$ in Theorem \ref{11:th3}.
\begin{corollary}\label{11:cor2}
\begin{itemize}
    \item[(i)] The double generating function for the probabilities of the run $F_{k_1}\prec F_{k_2}\prec \cdots \prec F_{k_m}$ is given by
    \begin{align*}
        \Phi_m(w,z)=\frac{\prod_{i=2}^{m-1}(1-p_i z)}{(1-z)\prod_{i=2}^{m-1}(1-p_iz)-(w-1)\prod_{i=1}^{m}(p_i z)^{\ell_i}(1-p_i z)}.
    \end{align*}
    \item[(ii)] The double generating function for the probabilities of the run $G_{k_1}\prec G_{k_2}\prec \cdots \prec G_{k_m}$ is given by
    \begin{align*}
      \Phi_m(w,z)=\frac{\prod_{i=2}^{m-1}(1-p_i z)}{(1-z)\prod_{i=2}^{m-1}(1-p_iz)-(w-1)\prod_{i=1}^{m}(p_i z)(1-(p_i z)^{k_i})}.
    \end{align*}
    \end{itemize}
\end{corollary}

\begin{remark}
\begin{itemize}
\item[(i)] Corollary \ref{11:cor1} and Corollary \ref{11:cor2} extend the patterns previously considered by Dafnis et al. \cite{DAP} (for two-state trials) to multi-state trials.
\item[(ii)] Using a similar argument as previously discussed ($k_i\to \infty$, $\ell_i=k_i$, and $\ell_i=1$, for $i\subseteq\{1,2,\ldots,m\}$), we can derive the generating function for various types of $(k_1,k_2,\ldots,k_m)$-runs with all cases of at least, exactly, and at most occurrences.
\end{itemize}
\end{remark}

\noindent
Next, let $\alpha_j$ denote the coefficient of $z^j$ in $\prod_{i=2}^{m-1}(1-p_iz)$, which is easy to compute given the value of $m$. It is important to note that $\alpha_0=1$. Therefore, we have
\begin{align}
    \prod_{i=2}^{m-1}(1-p_iz)=1+\sum_{j=1}^{m-2}\alpha_j z^j.\label{11:41}
\end{align}
Next, using Theorem \ref{11:th3} and the definition of double generating function given in \eqref{11:dgf}, the following corollary can be easily verified.

\begin{corollary}\label{11:cor3}
    The recursive relation in the probability generation function $\phi_{m,n}(\cdot)$ is given by
    \begin{align*}
        \phi_{m,n}(w)&+\sum_{j=1}^{m-2}(\alpha_j-\alpha_{j-1})\phi_{m,n-j}(w)-\alpha_{m-2} \phi_{m,n-m+1}(w)\\
        &=(w-1)\left(\prod_{i=1}^{m}p_i^{\ell_i}\right)\phi_{m,n-\ell}(w)-(w-1)\left(\prod_{i=1}^{m}p_i^{k_i+1}\right)\phi_{m,n-k-m}(w), \quad \text{for }n\ge \ell
    \end{align*}
    with initial condition $\phi_{m,n}(w) = 1$, for $n < \ell$ where $\ell:=\ell_1+\ell_2+\cdots+\ell_m$, $k:=k_1+k_2+\cdots+k_m$.
\end{corollary}

\noindent
Using Corollary \ref{11:cor3}, we can easily derive the following result.
\begin{corollary}
    The recursive relation in the probability mas function $p_{m,n}(\cdot)$ is given by
    \begin{align*}
       p_{m,n}(s)&+\sum_{j=1}^{m-2}(\alpha_j-\alpha_{j-1})p_{m,n-j}(s)-\alpha_{m-2} p_{m,n-m+1}(s)\\
        &=\left(\prod_{i=1}^{m}p_i^{\ell_i}\right)(p_{m,n-\ell}(s-1)-p_{m,n-\ell}(s))-\left(\prod_{i=1}^{m}p_i^{k_i+1}\right)(p_{m,n-k-m}(s-1)-p_{m,n-k-m}(s)),
    \end{align*}
     for $s\ge 0$ and $n\ge \ell$ with initial condition $p_{m,0}(0) = 1$ and $p_{m,n}(s)=0$ for $s>0$ and $n < \ell$.
\end{corollary}

\noindent
Next, let $\mu_{n,m,r}=\mathbb{E}\big(\big(X_{m;\ell_1:k_1,\ell_2:k_2,\ldots,\ell_m:k_m}^{(n)}\big)^r\big)$, the $r$th non-central moment of $X_{m;\ell_1:k_1,\ell_2:k_2,\ldots,\ell_m:k_m}^{(n)}$. The following result provide the recursive relation in $\mu_{n,m,r}$.

\begin{corollary}
    The moments $\mu_{n,m,r}$ satisfies the following recursive relation
    \begin{align*}
       \mu_{m,n,r}&+\sum_{j=1}^{m-2}(\alpha_j-\alpha_{j-1})\mu_{m,n-j,r}-\alpha_{m-2} \mu_{m,n-m+1,r}\\
        &=\left(\prod_{i=1}^{m}p_i^{\ell_i}\right)\sum_{u=0}^{r-1}\binom{r}{u}\mu_{m,n-\ell,u}-\left(\prod_{i=1}^{m}p_i^{k_i+1}\right)\sum_{u=0}^{r-1}\binom{r}{u}\mu_{m,n-k-m,u}, \quad \text{for }r\ge 1 \text{ and }  n\ge \ell
    \end{align*}
    with initial condition $\mu_{m,n,0} = 1$ and $\mu_{m,n,r}=0$ for $r\ge 1$ and $n < \ell$.
\end{corollary}

\begin{remark}
    Note that if we are interested in a specific pattern related to at least, exactly, and at most cases then we can get the pattern by applying $k_i\to \infty$, $\ell_i=k_i$, and $\ell_i=1$ for $i\subseteq \{1,2,\ldots,m\}$. Consequently, all the results mentioned above can be derived accordingly for that particular pattern. Here, observe that the notation $i\subseteq \{1,2,\ldots,m\}$ indicates that it corresponds to the pattern under consideration. For instance, if we are dealing the pattern $E_{\ell_1}\prec F_{\ell_2}\prec E_{\ell_3}\prec E_{\ell_4,k_4}\prec G_{k_5}$, then we set $k_1,k_3\to \infty$, $\ell_2=k_2$, and $\ell_5=1$ to derive the relevant results.
\end{remark}

\section{Waiting time Distribution of $\boldsymbol{X_{m;\ell_1:k_1,\ell_2:k_2,\ldots,\ell_m:k_m}^{(n)}}$}\label{11:sec3}
In this section, our focus lies on determining the waiting time distribution for the pattern \textbf{(R)}. We adopt a similar technique used by Chadjiconstantinidis and Eryilmaz \cite{CE2023} to derive the results. Recall that $\xi_1,\xi_2,\ldots,\xi_n$ is a sequence of \textit{iid} multinomial random variables, each taking values from the set $\{1,2,\ldots,m\}$, with probability $\mathbb{P}(\xi_i=j)=p_j$, for $i=1,2,\ldots,n$, $j=1,2,\ldots,m$, and $\sum_{j=1}^{m}p_j=1$. The random variable $X_{m;\ell_1:k_1,\ell_2:k_2,\ldots,\ell_m:k_m}^{(n)}$ denotes the number of occurrences of the run \textbf{(R)} with $p_{m,n}(s)=\mathbb{P}(X_{m;\ell_1:k_1,\ell_2:k_2,\ldots,\ell_m:k_m}^{(n)}=s)$. Further,  let $T_r$ denote the waiting time for the $r$th occurrence of the run \textbf{(R)}. Thus, we express $T_r$ as follows:
\begin{align}
    T_r=\sum_{i=1}^{r}W_i,\label{11:eq31}
\end{align}
where $W_1=T_1$ is the waiting time for the first occurrence of the pattern \textbf{(R)} and $W_i=T_i-T_{i-1}$, $i=2,3,\ldots,r$, are the inter-arrival time. Clearly, $T_r\ge r(\ell_1+\ell_2+\cdots+\ell_m)$. It is important to observe that since the underlying sequence consists of \textit{iid} multi-state random variables, $W_i$'s are also \textit{iid} random variables. It is easy to verified that
\begin{align*}
    \mathbb{P}(T_r>n)=\mathbb{P}(X_{m;\ell_1:k_1,\ell_2:k_2,\ldots,\ell_m:k_m}^{(n)}<r).
\end{align*}
However, computing waiting time probabilities using the above expression is practically challenging. Therefore, we employ double generating function to compute the probability generating function (pgf) of $T_r$. Let us define
\begin{align*}
    \psi_r(z):=\sum_{s=0}^{\infty} h_r(s)z^s=\sum_{s=0}^{\infty} \mathbb{P}(T_r=s)z^s,
\end{align*}
Following the steps similar to the proof of Lemma 1 of Chadjiconstantinidis and Eryilmaz \cite{CE2023}, the following lemma can be easily verified.

\begin{lemma}
    Let $\psi_1$ represent the pgf for the waiting time random variable $T_1$. Then
    \begin{align*}
        \psi_1(z)=1-(1-z)\Phi_m(0,z)=\frac{\prod_{i=1}^{m}((p_iz)^{\ell_i}-(p_i z)^{k_i+1})}{(1-z)\prod_{i=2}^{m-1}(1-p_i z)+\prod_{i=1}^{m}((p_iz)^{\ell_i}-(p_i z)^{k_i+1})}.
    \end{align*}
\end{lemma}

\noindent
It is known that $W_i$'s are \textit{iid} random variables.  Therefore, using \eqref{11:eq31}, the following theorem can be readily established.

\begin{theorem}\label{11:th4}
    For $r\ge 1$, the pgf of $T_r$ is given by
    \begin{align}
        \psi_r(z)=\left(\frac{\prod_{i=1}^{m}((p_iz)^{\ell_i}-(p_i z)^{k_i+1})}{(1-z)\prod_{i=2}^{m-1}(1-p_i z)+\prod_{i=1}^{m}((p_iz)^{\ell_i}-(p_i z)^{k_i+1})}\right)^r.\label{11:eq42}
    \end{align}
\end{theorem}

\begin{remark}
    If we allow $k_i\to \infty$, for $i=1,2,\ldots,m$ then the pgf of $T_r$ for the pattern $E_{\ell_1}\prec E_{\ell_2}\prec \cdots \prec E_{\ell_m}$ can be expressed as 
    \begin{align*}
        \psi_r(z)=\left(\frac{\prod_{i=1}^{m}(p_iz)^{\ell_i}}{(1-z)\prod_{i=2}^{m-1}(1-p_i z)+\prod_{i=1}^{m}(p_iz)^{\ell_i}}\right)^r,
    \end{align*}
   which coincides with Theorem 1 of Chadjiconstantinidis and Eryilmaz \cite{CE2023}, as expected.
\end{remark}

\noindent
Further, by substituting $\ell_i=k_i$ and $\ell_i=1$ for $i=1,2,\ldots,m$ in Theorem \ref{11:th4}, we obtain the following corollary.

\begin{corollary}
    \begin{itemize}
        \item[(i)] For $r\ge 1$, the pgf of $T_r$ for the pattern $F_{k_1}\prec F_{k_2}\prec \cdots \prec F_{k_m}$ is given by
    \begin{align*}
        \psi_r(z)=\left(\frac{\prod_{i=1}^{m}(p_iz)^{k_i}(1-p_iz)}{(1-z)\prod_{i=2}^{m-1}(1-p_i z)+\prod_{i=1}^{m}(p_iz)^{k_i}(1-p_iz)}\right)^r.
    \end{align*}
    \item[(i)] For $r\ge 1$, the pgf of $T_r$ for the pattern $G_{k_1}\prec G_{k_2}\prec \cdots \prec G_{k_m}$ is given by
    \begin{align*}
        \psi_r(z)=\left(\frac{\prod_{i=1}^{m}(p_iz)(1-(p_iz)^{k_i})}{(1-z)\prod_{i=2}^{m-1}(1-p_i z)+\prod_{i=1}^{m}(p_iz)(1-(p_iz)^{k_i})}\right)^r.
    \end{align*}
    \end{itemize}
\end{corollary}

\noindent
Using \eqref{11:41} in \eqref{11:eq42}, the following corollary can be easily obtained.

\begin{corollary}
\begin{itemize}
    \item[(i)] The probability mass function of $T_r$ satisfies the following recursive relation
    \begin{align*}
        h_r(s)&+\sum_{j=1}^{m-1}(\alpha_j-\alpha_{j-1})h_r(s-j)-\alpha_{m-2}h_r(s-m+1)+\prod_{i=1}^{m}p_i^{\ell_i}(h_r(s-\ell)-h_{r-1}(s-\ell))\\
        &=\prod_{i=1}^{m}p_i^{k_i+1}(h_r(s-k-m)-h_{r-1}(s-k-m))-\prod_{i=1}^{m}p_i^{\ell_i}(h_r(s-\ell)-h_{r-1}(s-\ell)),
    \end{align*}
    for $s\ge r\ell$, with initial conditions $h_0(s)=\delta_{s,0}$ and $h_r(s)=0$, for $s<r\ell$. 
    \item[(ii)] Let $\mu_{r,n}=\mathbb{E}((T_r)^n)$. Then $\mu_{r,n}$ satisfies the following recursive relation
    \begin{align*}
        \mu_{r,n}&-\sum_{u=0}^{n}\binom{n}{u}\mu_{r,u}+\sum_{j=1}^{m-2}\frac{n!}{(n-j)!} \alpha_j \mu_{r,n-j}-\sum_{j=1}^{m-2} \frac{n!}{(n-j)!} \alpha_j \sum_{u=0}^{n}\binom{n}{u}\mu_{r,n-j}\\
        &=\prod_{i=1}^{m}p_i^{k_i+1} \sum_{u=0}^{n}\binom{n}{u}(m+k)^{n-u}(\mu_{r,u}-\mu_{r-1,u})- \prod_{i=1}^{m}p_i^{\ell_i} \sum_{u=0}^{n}\binom{n}{u}\ell^{n-u}(\mu_{r,u}-\mu_{r-1,u})
    \end{align*}
    with initial condition $\mu_{0,n}=\delta_{n,0}$. Here, $\ell=\ell_1+\ell_2+\cdots+\ell_m$, $k=k_1+k_2+\cdots+k_m$, and $\delta_{s,0}$ is the Kronecker delta function.
    \end{itemize}
\end{corollary}

\begin{remark}
It is worth to note that if our interest lies in specific patterns pertaining to at least, exactly, and at most cases, we can obtain these patterns by setting $k_i\to \infty$, $\ell_i=k_i$, and $\ell_i=1$ for $i\subseteq \{1,2,\ldots,m\}$. Consequently, all the results mentioned above can be derived accordingly for that particular pattern. 
\end{remark}

\section{An Application to DNA Frequent Patterns}\label{11:sec4}
DNA sequencing is an important tool used to determine the nucleic acid sequence, that is, the order of nucleotides in DNA. The sequence consists of four letters: A, C, G, and T, which represent the nucleotides adenine, cytosine, guanine, and thymine, respectively. DNA frequent patterns refer to specific sub-patterns in DNA that occur more frequently than a manually set minimum support degree. For example, CG is a frequent pattern in the DNA sequence \{C, G, A, T, C, G\} when the minimum support degree is set to 2. If we increase the minimum support degree to more than 4 and aim to find the distribution of occurrences of DNA frequent patterns according to the studied patterns in this paper, then the distribution of $(k_1, k_2, k_3, k_4)$-runs becomes useful for studying the distribution of such DNA frequent patterns. For more details, see Deng {\em et al.} \cite{DCLX2019}, Yildiz and Selale \cite{BH2011}, and references therein.\\
To illustrate this, let us consider the occurrences of A, C, G, and T in a DNA sequence with probabilities $p_1$, $p_2$, $p_3$, and $p_4$, respectively. Suppose $n=50$ and the specific DNA frequent pattern is ``ACCGT'', that is, a $(1,2,1,1)$-run with exact occurrences. From Corollary 2, the double generating function of the considered runs is given by:
\begin{align}
\Phi_4(w,z)=\frac{\prod_{i=2}^{3}(1-p_i z)}{(1-z)\prod_{i=2}^{3}(1-p_iz)-(w-1)\prod_{i=1}^{4}(p_i z)^{k_i}(1-p_i z)},\label{11:exp1}
\end{align} 
where $k_1=1, k_2=2, k_3=k_4=1$. Note here that the minimum support degree is set to 5. Hence, the probability that the specific DNA frequent pattern occurs more than five times can be calculated using this expression. For instance, if $p_1=0.1$, $p_2=0.3$, $p_3=0.2$, and $p_4=0.4$, then the probabilities for the occurrences of the specific DNA frequent pattern are as follows:
\begin{center}
\begin{tabular}{|l|l|l|}
\hline
$m$ & $\mathbb{P}(X_{4;1,2,1,1}^{50}=m)$\\
\hline
$5$ & $7.17960 \times 10^{-9}$\\
\hline
$6$ & $2.68519 \times 10^{-11}$\\
\hline
$7$ & $4.73516 \times 10^{-14}$\\
\hline
$8$ & $3.03176 \times 10^{-17}$\\
\hline
$9$ & $3.78619 \times 10^{-21}$\\
\hline
$10$ & $6.34034 \times 10^{-27}$\\
\hline
\end{tabular}
\end{center}
As expected, the probability is very small as we are considering the occurrences of specific patterns more than 4 times. Similarly, different types of patterns used in various applications can be analysed using this approach.

\section*{Data Availability}
No data was used for the research described in the article.

\section*{Acknowledgements}
The first author is financially supported by TA fellowship from IIT (BHU) Varanasi, India. The second author is partially supported by SERB Start-up Research Grant (File No. SRG/2023/000236) and MATRICS Research Grant (File No. MTR/2023/000200), India.

%\setstretch{1.3}
%\singlespacing
%\small
%\footnotesize
\bibliographystyle{PV}
\bibliography{PV}

\end{document}